\documentclass[12pt]{amsart}
\usepackage{amsmath,amsxtra,amstext,amssymb,latexsym,amsthm,dsfont,amscd,xypic,cite,setspace}

\theoremstyle{plain}
\newtheorem{theorem}{Theorem}[section]
\newtheorem{thm}{Theorem}[section]
\newtheorem{lemma}[thm]{Lemma}
\newtheorem{corollary}[thm]{Corollary}
\newtheorem{proposition}[thm]{Proposition}
\newtheorem{definition}[thm]{Definition}

\newcommand{\gothG}{\mathfrak{g}}
\newcommand{\gothQ}{\mathfrak{q}}
\newcommand{\gothB}{\mathfrak{b}}
\newcommand{\gothH}{\mathfrak{h}}
\newcommand{\gothP}{\mathfrak{p}}
\newcommand{\gothL}{\mathfrak{l}}
\newcommand{\codim}{{\rm codim}}

\newcommand{\coker}{{\rm coker\, }}

\newcommand{\Lie}{{\rm Lie}}
\newcommand{\tilG}{\tilde{G}}
\newcommand{\tilH}{\tilde{H}}
\newcommand{\tilB}{\tilde{B}}
\newcommand{\tilP}{\tilde{P}}
\newcommand{\tilR}{\tilde{R}}

\newcommand{\tilQ}{\tilde{Q}}
\newcommand{\tilW}{\tilde{W}}

\newcommand{\Te}{T_{eP}}

\begin{document}

\title[A multiplicative formula for structure constants]{A multiplicative formula for structure constants in the cohomology of flag varieties}

\author{Edward Richmond}

\address{Department of Mathematics, University of British Columbia, Vancouver, BC V6T 1Z2, Canada}
\email{erichmond@math.ubc.ca}


\begin{abstract}
Let $G$ be a complex semisimple Lie group and let $P,Q$ be a pair of parabolic subgroups of $G$ such that $Q$ contains $P$.  Consider the flag varieties $G/P$, $G/Q$ and $Q/P$. We show that certain structure constants in $H^*(G/P)$ with respect to the Schubert basis can be written as a product of structure constants coming from $H^*(G/Q)$ and $H^*(Q/P)$ in a very natural way. The primary application is to compute Levi-movable structure constants defined by Belkale and Kumar in \cite{BK06}.  We also give a generalization of this product formula in the branching Schubert calculus setting.
\end{abstract}





\maketitle

\section{Introduction}

Let $G$ be a connected, simply connected, semisimple complex algebraic group and let $P\subseteq Q$ be a pair of parabolic subgroups.  Consider the induced sequence of flag varieties \begin{equation}\label{flag sequence}Q/P\hookrightarrow G/P\twoheadrightarrow G/Q.\end{equation}
The goal of this paper is to give a simple multiplicative formula connecting the structure coefficients for the cohomology ring of the three flag varieties in \eqref{flag sequence} with respect to their Schubert bases.  Let $W$ be the Weyl group of $G$ and let $W_P\subseteq W_Q\subseteq W$ denote the Weyl groups of $P$ and $Q$ respectively.  Let $W^P\subseteq W$ denote the set of minimal length coset representatives in $W/W_P.$  For any $w\in W^P$, let $\bar X_w\subseteq G/P$ denote the corresponding Schubert variety and let $[X_w]\in H^*(G/P)=H^*(G/P,\mathds{Z})$ denote the Schubert class of $\bar X_w$.  It is well known that the Schubert classes $\{[X_w]\}_{w\in W^P}$ form an additive basis for cohomology.  Similarly, we have Schubert classes $[X_u]\in H^*(G/Q)$ for $u\in W^Q$ and $[X_v]\in H^*(Q/P)$ for $v\in W^P\cap W_Q$.  The letters $w, u, v$ will be used to denote Schubert varieties in $G/P$, $G/Q$ and $Q/P$ respectively.   In Lemma \ref{lemma_wuv}, we show that for any $w\in W^P$, there is a unique decomposition $w=uv$ where $u\in W^Q$ and $v\in W^P\cap W_Q$.  Fix $s\geq 2$ and for any $w_1,\ldots,w_s\in W^P$ such that $\sum_{k=1}^s\codim\ X_{w_k}=\dim G/P$, define the associated structure coefficient (or structure constant) to be the integer $c_w$ where $$[X_{w_1}]\cdots[X_{w_s}]=c_w[pt]\in H^*(G/P).$$  The following is the first result of this paper:

\begin{theorem}\label{Thm1}Let $w_1, \ldots, w_s\in W^P$ and let $u_k\in W^Q, v_k\in W^P\cap W_Q$ be defined by $w_k=u_kv_k$.  Assume that \begin{equation}\label{codim condition}\sum_{k=1}^s\codim\ X_{w_k}=\dim G/P\quad \mbox{and}\quad \sum_{k=1}^s\codim\ X_{u_k}=\dim G/Q.\end{equation}  If $c_w,c_u,c_v\in\mathds{Z}_{\geq 0}$ are defined by:
$$\prod_{k=1}^s[X_{w_k}]=c_w[pt],\quad \prod_{k=1}^s[X_{u_k}]=c_u[pt],\quad \prod_{k=1}^s[X_{v_k}]=c_v[pt]$$ in $H^*(G/P), H^*(G/Q), H^*(Q/P)$ respectively, then $c_w=c_u\cdot c_v$. \end{theorem} Note that the dimensional conditions in \eqref{codim condition} imply that $\sum_{k=1}^s\codim\ X_{v_k}=\dim Q/P$ and hence the associated structure constant $c_v$ is well defined.

\smallskip

To prove Theorem \ref{Thm1}, we study the geometry of \eqref{flag sequence}.  Fix a maximal torus $H$ and and Borel subgroup $B$ such that $H\subseteq B\subseteq P$.  It is well known that if $\prod_{k=1}^s[X_{w_k}]=c_w[pt]$, then the number of points in the intersection of generic translates \begin{equation}\label{intersection size}\left|g_1X_{w_1}\cap\cdots\cap g_sX_{w_s}\right|=c_w.\end{equation}  We show that for a generic choice of $(\bar g_1,\ldots,\bar g_s)\in (G/B)^s$, the intersection given in \eqref{intersection size} projects onto the intersection $\bigcap_{k=1}^sg_kX_{u_k}\subseteq G/Q$ with each fiber of the projection containing exactly $c_v$ points.  The techniques used in the proof are inspired by Belkale's work in \cite{Be06}.

\subsection{Levi-movability}\label{section Lmov}

The main application of Theorem \ref{Thm1} is to show that the product formula applies to``Levi-movable" $s$-tuples $(w_1,\ldots,w_s)\in (W^P)^s$.  Let $L_P$ denote the Levi subgroup of $P$ containing $H$.  Belkale and Kumar give the following definition in \cite{BK06}.

\begin{definition}The $s$-tuple $(w_1,\ldots,w_s)\in(W^P)^s$ is Levi movable or $L_P$-movable if $$\sum_{k=1}^s\codim\ X_{w_k}=\dim G/P$$ and for generic $(l_1,\ldots,l_s)\in (L_P)^s$ the intersection $$l_1w_1^{-1}X_{w_1}\cap\cdots\cap l_sw_s^{-1}X_{w_s}$$ is transverse at $eP\in G/P$.\end{definition}

If $(w_1,\ldots, w_s)$ is Levi-movable, then the associated structure constant is not zero. These statements become equivalent if we also assume that $(w_1,\ldots, w_s)$ satisfies a system of linear equalities given in \cite[Theorem 15 (b)]{BK06} (these equalities are also given later in Proposition \ref{Bel_Ku Thm}).  The following is the second result of this paper.

\begin{theorem}\label{Thm2}Let $(w_1,\ldots,w_s)$ be $L_P$-movable and let $u_k\in W^Q, v_k\in W^P\cap W_Q$ be defined by $w_k=u_kv_k$.  The following are true:
\begin{enumerate}\item[(i)]$(u_1,\ldots,u_s)$ is $L_Q$-movable \item[(ii)] $(v_1,\ldots,v_s)$ is $L_{(L_Q\cap P)}$-movable.\end{enumerate}\end{theorem}

As a consequence of Theorem \ref{Thm2}, if $(w_1,\ldots,w_s)$ is $L_P$-movable, we can apply the product formula in Theorem \ref{Thm1} to its associated structure constant since the conditions in \eqref{codim condition} are satisfied.  Moreover, since $(u_1,\ldots,u_s)$ and $(v_1,\ldots,v_s)$ are also Levi-movable, we can again apply the product formula to decompose their associated structure constants. This reduces the problem of computing structure constants associated to any Levi movable $s$-tuple to those coming from the cohomology of flag varieties $G/P$ where $P$ is maximal parabolic subgroup of $G$.

\smallskip

The author has proved a special case of Theorems \ref{Thm1} and \ref{Thm2} for type A flag varieties in \cite[Theorem 3]{Ri08} and type C flag varieties in his thesis \cite{RiThesis}.  The techniques used to prove Theorem \ref{Thm1} are direct generalizations of the those used in \cite{Ri08, RiThesis}.  However, the proof of Theorem \ref{Thm2} is different that the proof for the type A and C cases in previous papers.  The results of Theorems \ref{Thm1} and \ref{Thm2} were also obtained simultaneously by Ressayre in \cite{Re208}.  We remark that Ressayre's proof of these Theorems is different that those presented in this paper.

\smallskip

Unfortunately, the converse to Theorem \ref{Thm2} is false.  Counter examples already exist for two-step flag varieties of type A.  In the following corollary, we give a ``numerical" converse which can be recovered if we assume $(w_1,\ldots,w_s)$ satisfies the numerical conditions given in \cite[Theorem 15 (b)]{BK06}. These conditions are also stated in Proposition \ref{Bel_Ku Thm} of this paper.

\begin{corollary}\label{Cor_converse}Let $(w_1,\ldots,w_s)\in(W^P)^s$ and let $u_k\in W^Q, v_k\in W^P\cap W_Q$ be defined by $w_k=u_kv_k$.  Assume that the following are true:
\begin{enumerate}
\item[(i)]$(u_1,\ldots,u_s)$ is $L_Q$-movable
\item[(ii)] $(v_1,\ldots,v_s)$ is $L_{(L_Q\cap P)}$-movable.
\item[(iii)]$(w_1,\ldots,w_s)$ satisfies the numerical conditions given in Proposition \ref{Bel_Ku Thm}.\end{enumerate}
Then $(w_1,\ldots,w_s)$ is $L_P$-movable.\end{corollary}

This corollary is a direct consequence of Theorem \ref{Thm1} and \cite[Theorem 15]{BK06}.  We remark that Corollary \ref{Cor_converse} can also be established by work outside this paper. In particular, \cite[Theorem 15]{BK06} and \cite[Proposition 11]{PS06} would also imply Corollary \ref{Cor_converse}.

\subsection{Representation theory and tensor product invariants}\label{Section Rep inv}

In this section we state a corollary of Theorems \ref{Thm1} and \ref{Thm2} in regards to representation theory of the group $G$.  Let $X(H)$ denote the character group of the maximal torus $H$ and let $X^+(H)$ denote the set of dominant characters with respect to the Borel subgroup $B$.  For any dominant character $\lambda\in X^+(H)$ of $G$, let $V_{\lambda}$ denote the corresponding irreducible finite dimensional representation of $G$ of highest weight $\lambda$.  For any $s\geq 2$, define $$\Gamma(s,G):=\{(\lambda_1,\ldots, \lambda_s)\in X^+(H)^s\otimes_{\mathds{Z}} \mathds{Q}\ |\ \exists\ N>1,\ (V_{N\lambda_1}\otimes\cdots\otimes V_{N\lambda_s})^G\neq 0\}.$$  The set $\Gamma(s,G)$ forms a convex cone in the vector space $X^+(H)^s\otimes_{\mathds{Z}}\mathds{Q}$ and has been studied in the context of Horn's problem on generalized eigenvalues \cite{BK06,Fu00,Ho62}.  The set $\Gamma(s,G)$ was initially characterized by Klyachko \cite{Kly98} in type A and later in all types by Berenstein and Sjamaar \cite{BS00}.  These characterizations consist of a list of inequalities parameterized by nonzero products of Schubert classes.  In \cite{KTW04}, Knutson, Tao and Woodard determined a minimal set of inequalities for type A.  In \cite{BK06}, Belkale and Kumar showed that for all types, it is enough to consider the set of inequalities corresponding to Levi-movable $s$-tuples with associated structure coefficient equal to 1.  Most recently,  Ressayre \cite{Re07} showed that this set of inequalities is in fact minimal. Let $\Delta$ denote the set of simple roots of $G$ and let $\Delta(P)$ denote the simple roots associated to the parabolic subgroup $P\subseteq G$.  For any $\alpha\in\Delta$, let $\omega_{\alpha^{\vee}}$ denote the corresponding fundamental coweight.

\begin{theorem}\label{Ressayre}(Belkale and Kumar \cite{BK06}, Ressayre \cite{Re07}) If $(w_1,\ldots,w_s)\in W^P$ is $L_P$-movable with associated structure constant $c_w=1$, then the set of $(\lambda_1,\ldots,\lambda_s)\in\Gamma(s,G)$ such that $$\sum_{k=1}^s\omega_{\alpha^{\vee}}(w_k^{-1}\lambda_k)=0\quad \forall\ \alpha\in \Delta\backslash\Delta(P)$$ is a face of $\Gamma(s,G)$ whose codimension is of cardinality $|\Delta\backslash\Delta(P)|$.  Moreover, any face of $\Gamma(s,G)$ which intersects the interior of the dominant chamber $X^+(H)^s\otimes_{\mathds{Z}} \mathds{Q}$ can be described as above, and the list of faces of codimension 1 is irredundant.\end{theorem}

Let $F(w_1,\ldots,w_s)\subseteq\Gamma(s,G)$ be the face of $\Gamma(s,G)$ associated to the Levi movable $s$-tuple $(w_1,\ldots,w_s)\in (W^P)^s$ with $c_w=1$.  Applying Theorems \ref{Thm1} and \ref{Thm2} yields the following corollary:

\begin{corollary}Let $(w_1,\ldots,w_s)\in (W^P)^s$ be $L_P$-movable with associated structure constant $c_w=1$ and let $w_k=u_kv_k$ where $u_k\in W^Q$ and $v_k\in W^P\cap W_Q$.  Then $F(w_1,\ldots,w_s)$ is a face of $F(u_1,\ldots,u_s)$. \end{corollary}

\begin{proof}By Theorems \ref{Thm1} and \ref{Thm2}, we have that $(u_1,\ldots, u_s)$ is $L_Q$-movable and that $c_w=c_u\cdot c_v=1$, where $c_w,c_u,c_v$ are the structure constants associated to $(w_k)_{k=1}^s, (u_k)_{k=1}^s, (v_k)_{k=1}^s$ respectively.  Hence $c_u=1$ and by Theorem \ref{Ressayre}, $F(u_1,\ldots,u_s)$ is a face of $\Gamma(s,G)$ of codimension $|\Delta\backslash\Delta(Q)|$.  It suffices to show that if $(\lambda_1,\ldots,\lambda_s)\in F(w_1,\ldots,w_s)$, then $(\lambda_1,\ldots,\lambda_s)\in F(u_1,\ldots,u_s)$. Let $\alpha\in \Delta\backslash\Delta(Q)\subseteq\Delta\backslash\Delta(P)$.  Then for any $w\in W^P$ and rational dominant weight $\lambda$, we have $$\omega_{\alpha^{\vee}}(w^{-1}\lambda)=uv\omega_{\alpha^{\vee}}(\lambda)=u\omega_{\alpha^{\vee}}(\lambda)=\omega_{\alpha^{\vee}}(u^{-1}\lambda)$$ since $v\in W_Q$ acts trivially on any $\omega_{\alpha^{\vee}}$ where $\alpha\in\Delta\backslash\Delta(Q)$.  This proves the corollary.\end{proof}

\subsection{Generalizations to branching Schubert calculus}

In this section, we give generalizations of Theorems \ref{Thm1} and \ref{Thm2}.  We remark that the generalization of Theorem \ref{Thm1} was also independently obtained by Ressayre in \cite{Re208}.  Let $\tilG$ be any connected semisimple subgroup of $G$ and fix maximal tori and Borel subgroups $\tilH\subseteq\tilB\subseteq\tilG$ and $H\subseteq B\subseteq G$ such that $\tilH=H\cap\tilG$ and $\tilB=B\cap\tilG$.  For any parabolic subgroup $P\subseteq G$ containing $B$, we define parabolic subgroup $\tilP:=P\cap\tilG$ of $\tilG$.  Consider the $\tilG$-equivariant embedding of flag varieties $$\phi_z:\tilG/\tilP\hookrightarrow G/P$$ defined by $\phi(g\tilP):=gP$.  The problem concerning ``branching Schubert calculus" is to compute the pullback $$\phi^*([X_w])=\sum_{\tilde w\in \tilW^P}c_w^{\tilde w}[X_{\tilde w}]$$ in terms of the Schubert basis in $H^*(\tilG/\tilP)$.  If $\dim X_w=\dim G/P-\dim \tilG/\tilP$, then $\phi^*([X_w])=c_w[pt]$ for some $c_w\in\mathds{Z}_{\geq 0}$.  Consider the diagonal embedding $\phi:\tilG/\tilP\hookrightarrow (\tilG/\tilP)^s$ and let $[X_{w_1}\times\cdots\times X_{w_s}]$ be a Schubert class in $H^*((\tilG/\tilP)^s)$.  We have that $$\phi^*([X_{w_1}\times\cdots\times X_{w_s}])=\prod_{k=1}^s[X_{w_k}].$$  Hence the problem of branching Schubert calculus is a generalization of usual Schubert calculus.

\smallskip

Let $Q$ be a parabolic subgroup which contains $P$ and define $\tilQ:=Q\cap\tilG$ to be the corresponding parabolic subgroup of $\tilG$.  The embedding $\phi$ induces the maps $$\phi_1:\tilG/\tilQ\hookrightarrow G/Q\quad \mbox{and}\quad \phi_2:\tilQ/\tilP\hookrightarrow Q/P$$ given by $\phi_1(g\tilQ):=gQ,$ and $\phi_2:=\phi|_{\tilQ/\tilP}.$  The following is an analogue of Theorem \ref{Thm1}:

\begin{theorem}\label{Thm1branching}
Let $w=uv\in W^P$ where $u\in W^Q$ and $v\in W^P\cap W_Q$.  Assume that \begin{equation}\label{codim condition2}\dim X_w=\dim G/P-\dim \tilG/\tilP\quad \mbox{and}\quad \dim X_u=\dim G/Q-\dim \tilG/\tilQ.\end{equation}

If $c_w,c_u,c_v\in\mathds{Z}_{\geq 0}$ are defined by:
$$\phi^*([X_w])=c_w[pt],\quad \phi^*_1([X_u])=c_u[pt],\quad \phi^*_2([X_v])=c_v[pt]$$ in $H^*(\tilG/\tilP), H^*(\tilG/\tilQ), H^*(\tilQ/\tilP)$ respectively, then $c_w=c_u\cdot c_v$. \end{theorem} The techniques used to prove Theorem \ref{Thm1branching} are the same as those used to prove Theorem \ref{Thm1}, so we only provide a brief overview in Section \ref{Branching}.

As in Section \ref{section Lmov}, we give a special set of $w\in W^P$ which satisfy the assumptions in Theorem \ref{Thm1branching} by generalizing the notion of Levi-movability.
\begin{definition}We say $w\in W^P$ is $(L_P,\phi)$-movable if for generic $l\in L_P$ the following induced map on tangent spaces is an isomorphism:$$\phi_*:T_{e\tilP}(\tilG/\tilP)\rightarrow \frac{T_{eP}(G/P)}{T_{eP}(lw^{-1}X_w)}.$$\end{definition}
If $\phi$ is the diagonal embedding, then $w=(w_1,\ldots,w_s)$ is $(L_P,\phi)$-movable if and only if $w$ is $L_P$-movable.  We now give an analogue of Theorem \ref{Thm2}.  Let $\tilde\gothH$ denote the Lie algebra of $\tilH$ and let $\Delta_{\tilG}\subset\tilde\gothH^*$ denote the simple roots of $\tilG$.  Let $\Delta_{\tilQ}\subseteq\Delta_{\tilG}$ denote the set of simple roots corresponding to the parabolic subgroup $\tilQ\subseteq\tilG$.  Let $\mathfrak{Z}$ denote the Lie algebra of the center of $L_Q$.

\begin{theorem}\label{Thm2branching}Assume there exists a vector $\tau\in\tilde\gothH\cap \mathfrak{Z}$ such that $\alpha(\tau)\geq 0$ for any $\alpha\in\Delta_{\tilG}$ with equality if and only if $\alpha\in\Delta_{\tilQ}$.

Let $w=uv\in W^P$ where $u\in W^Q$ and $v\in W^P\cap W_Q$.  If $w$ is $(L_P,\phi)$-movable, then the following are true:
\begin{enumerate}\item[(i)] $u$ is $(L_Q,\phi_1)$-movable \item[(ii)] $v$ is $(L_{(L_Q\cap P)},\phi_2)$-movable.\end{enumerate}\end{theorem}

The existence of $\tau\in\tilde\gothH\cap \mathfrak{Z}$ in Theorem \ref{Thm2branching} is a restriction on the choice of $Q\subseteq G$.  In the case of the diagonal embedding, the vector $\tau$ exists if and only if the parabolic subgroup $Q\subseteq G=\tilG^s$ is of the form $Q=\tilQ^s$ for some parabolic subgroup $\tilQ\subset\tilG$.

\smallskip

Theorem \ref{Thm2branching} implies that if $w\in W^P$ is $(L_P,\phi)$-movable, then $w$ satisfies the conditions in \eqref{codim condition2} and hence we can decompose the associated structure constant $c_w$.  As with Theorem \ref{Thm1branching}, the proof of Theorem \ref{Thm2branching} follows the same outline as the proof in the diagonal embedding case.

\subsection*{Acknowledgments} I would like to thank Prakash Belkale for suggesting I investigate these product formulas.  I would also like to thank Shrawan Kumar for ideas and comments on Theorem \ref{Thm2}, pointing out that Theorems \ref{Thm1} and \ref{Thm2} could be generalized to the branching Schubert calculus setting and his overall generous input.  Finally, I would like to thank the referee for his/her valuable comments and suggestions.

\section{Preliminaries}

Let $G$ be a connected, simply connected, semisimple complex algebraic group.  Fix a Borel subgroup $B$ and a maximal torus $H\subseteq B$.  Let $W:=N_G(H)/H$ denote the Weyl group of $G$ where $N_G(H)$ is the normalizer of $H$ in $G$.  Let $P\subseteq G$ be a standard parabolic subgroup ($P$ contains $B$) and let $L_P$ denote the Levi subgroup of $P$ containing $H$.  Denote the Lie algebras of $G,H,B,P,L_P$ by the corresponding frankfurt letters $\gothG,\gothH,\gothB,\gothP,\gothL_P$.

\smallskip

Let $R\subseteq \gothH^*$ be the set of roots and let $R^{\pm}\subseteq R$ denote the set of positive roots (negative roots).  Let $R_{P}$ denote the set of roots corresponding to $\gothL_P$ and let $R_P^{\pm}$ denote the set of positive roots (negative roots) with respect to the Borel subgroup $B_P:=B\cap L_P$ of $L_P$.

\smallskip

Let $W^P$ be the set of minimal length representatives of the coset space $W/W_P$ where $W_P$ is the Weyl group of $P$ (or $L_P$).  For any $w\in W^P$, define the Schubert cell $$X_w:=BwP/P\subseteq G/P.$$  We denote the cohomology class of the closure $\bar X_w$ by $[X_w]\in H^*(G/P)$.  We begin with some basic lemmas on the Weyl group $W$.

\begin{lemma}\label{lemma_wuv} The map $\eta:W^Q\times(W^P\cap W_Q)\rightarrow W^P$ given by $(u,v)\mapsto uv$ is well defined and a bijection.\end{lemma}
\begin{proof}Since $W=\bigsqcup_{u\in W^Q} uW_Q$, we have that $W/W_P=\bigsqcup_{u\in W^Q} uW_Q/W_P$.  It suffices to show that if $v\in W^P\cap W_Q$, then $uv\in W^P$.  Let $\ell:W\rightarrow\mathds{Z}_{\geq 0}$ denote the length function on $W$.  For any $v'\in W_P$ we have that $$\ell(uvv')=\ell(u)+\ell(vv')=\ell(u)+\ell(v)+\ell(v')=\ell(uv)+\ell(v')$$ since $u\in W^Q$, $vv'\in W_Q$, $v\in W^P$ and $v'\in W_P$.  Hence $uv\in W^P$. \end{proof}

Lemma \ref{lemma_wuv} shows that for any $w\in W^P$, there is a unique $u\in W^Q$ and $v\in W^P\cap W_Q$ such that $w=uv$.  We will assume this relationship between $w,u,v$ given any $w\in W^P$.  If these groups elements are indexed $w_k\in W^P$, then we write $w_k=u_kv_k$ accordingly.

\smallskip

Note that the flag variety $Q/P\simeq L_Q/(L_Q\cap P)$ where $L_Q$ is the Levi subgroup of $Q$.  Under this identification, the Schubert cell $X_v\simeq B_Qv(L_Q \cap P)/(L_Q\cap P).$

\begin{lemma}\label{Lemma_intersection}For any $w=uv\in W^P$, we have that $u^{-1}X_w\cap Q/P=X_v.$\end{lemma}
\begin{proof}  Let $X_w'$ denote the subset of $L_Q/(L_Q\cap P)$ identified with $u^{-1}X_w\cap Q/P$ under the isomorphism $Q/P\simeq L_Q/(L_Q\cap P)$.  Since $v\in W_Q$, we have that $$u^{-1}X_w\cap Q/P=(u^{-1}BuvP\cap Q)P/P=(u^{-1}Bu\cap Q)vP/P.$$   By \cite[Excerise 1.3.E]{Ku02}, the group $B_Q\subseteq u^{-1}Bu\cap Q$ and hence $B_Qv(L_Q \cap P)/(L_Q\cap P)\subseteq X_w'$.  Since the $B_Q$-orbits of $L_Q/(L_Q\cap P)$ are in bijection with $W^P\cap W_Q$, the set $X_w'$ cannot contain more than a single $B_Q$-orbit. This proves the lemma\end{proof}

\section{Structure coefficients and transversality}

In this section we prove Theorem \ref{Thm1}.  Assume we have $(w_1,\ldots,w_s)\in (W^P)^s$ which satisfy the conditions \eqref{codim condition} and let $w_k=u_kv_k$ with respect to Lemma \ref{lemma_wuv}.  We begin by considering the following $G$-variety.  Define $$Y=Y(u_1,\ldots, u_s):=\{(\bar g\, ;\, \bar g_1,\ldots, \bar g_s)\in G/Q\times (G/B)^s\ |\ \bar g\in \bigcap_{k=1}^sg_kX_{u_k}\}$$ where the action $G$ on $Y$ is the diagonal action.  We now prove that $Y$ is smooth and irreducible.  Define $$\tilde Y:=G\times_Q (Qu_1^{-1}B/B\times\cdots\times Qu_s^{-1}B/B).$$  Note that if $\bar g\in g_kX_{u_k}$, then by \cite[Lemma 1]{BK06}, we have that $g^{-1}g_k=q_ku_k^{-1}$ for some $q_k\in Q.$  Since translated Schubert varieties of the form $qu_k^{-1}X_{u_k}$ are precisely those that contain the identity, we can think of $\tilde Y$ as the parameter set of all intersections $\bigcap_{k=1}^sg_kX_{u_k}$ which contain the identity up to translation paired with a point in $G.$

\begin{lemma}\label{Lemma Ysmooth}The $G$-equivariant map $\xi:\tilde Y\rightarrow Y$ given by
\begin{equation}\label{Ysmooth}\xi((g\, ;\, \overline{q_1u_1^{-1}}, \ldots, \overline{q_su_s^{-1}}))=(\bar g\, ;\, \overline{gq_1u^{-1}_1},\ldots, \overline{gq_su_s^{-1}}).\end{equation} is well defined and an isomorphism.  Moreover, $Y$ is smooth and irreducible.\end{lemma}

\begin{proof}If $\xi$ is an isomorphism, then the irreducibility and smoothness of $Y$ follows from the irreducibility of smoothness of $\tilde Y$.  The fact that $\xi$ is an isomorphism is a consequence of \cite[Lemma 6.1]{Re04}.\end{proof}

\begin{lemma}\label{lemma_QutoQ}For any $u\in W^Q$, the map $Qu^{-1}B/B\rightarrow Q/B$ given by $\overline{qu^{-1}}\mapsto\overline q$ is well defined and $Q$-equivariant. \end{lemma}

\begin{proof}Let $q_1,q_2\in Q$ such that $q_1u^{-1}B=q_2u^{-1}B$.  Then $uq_2^{-1}q_1u^{-1}\in B$.  It suffices to show that $q_2^{-1}q_1\in B$.  In other words, that $Q\cap u^{-1}Bu\subseteq B.$  By \cite[Proposition 2.1]{DM91}, the intersection $Q\cap u^{-1}Bu$ is connected and hence, it is enough to show that $\gothQ\cap u^{-1}\gothB\subseteq\gothB$.  We look at the set of roots $R_Q\cap u^{-1}R^+$ corresponding to $\gothQ\cap u^{-1}\gothB$.  Since $u\in W^Q$, we have that $uR_Q^+\subseteq R^+$ and $uR_Q^-\subseteq R^-$.  Thus $$R_Q\cap u^{-1}R^+=u^{-1}(uR_Q\cap R^+)=u^{-1}(uR^+_Q)\subseteq R^+.$$  This proves the lemma.\end{proof}

Assume we have $(\bar g\, ;\, \bar g_1,\ldots, \bar g_s)\in G/P\times (G/B)^s$ such that $\bar g\in\bigcap_{k=1}^sg_kX_{w_k}$.  It is easy to see that $(\bar gQ\, ;\, \bar g_1,\ldots, \bar g_s)\in Y$.  By \cite[Lemma 1]{BK06}, since $eP\in g^{-1}g_kX_{w_k}$, we have $g^{-1}g_kX_{w_k}=p_kv_k^{-1}u_k^{-1}X_{w_k}$ from some $p_k\in P$.  Set $q_k=p_kv_k^{-1}\in Q.$  By Lemma \ref{Lemma_intersection}, $$g^{-1}g_kX_{w_k}\cap Q/P=q_k(u_k^{-1}X_{w_k}\cap Q/P)=q_kX_{v_k}.$$  We consider the points of $Y$ that satisfy the following property.

\begin{definition}We say $(\bar g\, ;\, \bar g_1,\ldots, \bar g_s)\in Y$ has property $P1$ if:
\begin{enumerate}
\item $\displaystyle\bigcap_{k=1}^s (g^{-1}g_kX_{w_k}\cap Q/P)$ is transverse at every point in the intersection in $Q/P$
\item For any $(q_1\ldots,q_s)\in Q^s$ such that, $g^{-1}g_kX_{u_k}=q_ku_k^{-1}X_{u_k}\subseteq G/Q$ for all $k$, the intersection $$\displaystyle\bigcap_{k=1}^s q_kX_{v_k}=\bigcap_{k=1}^s q_k\bar X_{v_k}\subseteq Q/P.$$\end{enumerate}\end{definition}

\begin{proposition}\label{Prop P1 open}The set of points in $Y$ with property $P1$ contains a nonempty $G$-stable open subset.\end{proposition}
\begin{proof}By Kleiman's transversality \cite{Kl74}, there exists a nonempty open set $O\subseteq (Q/B)^s$ such that for any $(q_1,\ldots, q_s)\in O$ the following is satisfied:
\begin{enumerate}
\item $\bigcap_{k=1}^s q_kX_{v_k}\subseteq Q/P$ is transverse at every point in the intersection.
\item $\bigcap_{k=1}^s q_kX_{v_k}=\bigcap_{k=1}^s q_k\bar X_{v_k}$.\end{enumerate}

Moreover, we can choose $O$ to be stable under the diagonal action of $Q$ on $(Q/B)^s$ by replacing $O$ with $\bigcup_{q\in Q}qO$.  Consider the map $$\tilde\xi:Y\rightarrow G\times_Q (Q/B)^s$$ defined by $\tilde\xi:=\zeta\circ\xi^{-1}$ where $$\zeta((g\, ;\, \overline{q_1u_1^{-1}}, \ldots, \overline{q_su_s^{-1}})):=(g\, ;\, \overline{q_1}, \ldots, \overline{q_s}).$$  By Lemma \ref{lemma_QutoQ}, the map $\tilde\xi$ is well defined and $G$-equivariant.  Clearly any $(g\, ;\, g_1,\ldots,g_s)\in\tilde\xi^{-1}(G\times_Q O) $ satisfies property $P1$.\end{proof}

\subsection{Proof of Theorem \ref{Thm1}}

Assume that $c_u\neq 0$.  We first show there exists $(\bar g_1, \ldots \bar g_s)\in (G/B)^s$ which satisfies the following three conditions:
\begin{enumerate}
\item[(i)] $\bigcap_{k=1}^s g_kX_{w_k}$ is transverse at every point of the intersection in $G/P$ and $$\bigcap_{k=1}^s g_kX_{w_k}=\bigcap_{k=1}^s g_k\bar X_{w_k}.$$
\item[(ii)] $\bigcap_{k=1}^s g_kX_{u_k}$ is transverse at every point of the intersection in $G/Q$ and $$\bigcap_{k=1}^s g_kX_{u_k}=\bigcap_{k=1}^s g_k\bar X_{u_k}.$$
\item[(iii)] For every $x\in \bigcap_{k=1}^s g_kX_{u_k}$, we have that $(x\, ;\,\bar g_1,\ldots,\bar g_s)\in Y$ has property $P1$.
\end{enumerate}

By Kleiman's tranversality \cite{Kl74}, there exists an open subset $O_1\subseteq (G/B)^s$ such that every point in $O_1$ satisfies conditions (i) and (ii).  By Proposition \ref{Prop P1 open}, there exists a nonempty open subset $Y^{\circ}\subseteq Y$ such that every point in $Y^{\circ}$ has property $P1$.  Consider the projection of $Y$ onto its second factor $$\sigma:Y\rightarrow (G/B)^s.$$  Since $c_u\neq 0$, the morphism $\sigma$ is dominant.  Moreover, the fibers of $\sigma$ are generically finite and hence $\dim Y=\dim (G/B)^s$.  Since $Y$ is irreducible we have that $$\dim\overline{\sigma(Y\backslash Y^{\circ})}\leq\dim Y\backslash Y^{\circ}<\dim Y=\dim (G/B)^s.$$  Define the nonempty open set $O_2:=(G/B)^s\backslash(\overline{\sigma(Y\backslash Y^{\circ})})$.  Any $(\bar g_1, \ldots, \bar g_s)\in O_1\cap O_2$ satisfies conditions (i)-(iii).  Assume that $(\bar g_1, \ldots, \bar g_s)\in O_1\cap O_2\subseteq (G/B)^s$.  Conditions (i) and (ii) imply that $$\left|\bigcap_{k=1}^s g_kX_{w_k}\right|=c_w\ \mbox{and}\ \left|\bigcap_{k=1}^s g_kX_{u_k}\right|=c_u.$$
Consider the $G$-equivariant projection $\pi:G/P\twoheadrightarrow G/Q$.  If $\bar g\in \bigcap_{k=1}^s g_kX_{u_k}$, then condition (iii) implies that $(\bar g\, ;\,\bar g_1,\ldots,\bar g_s)\in Y$ has property $P1$.  By Lemma \ref{Lemma_intersection}, we have \begin{equation}\label{equation_cap}\left|\bigcap_{k=1}^s g_kX_{w_k}\cap\pi^{-1}(\bar g)\right|=\left|\bigcap_{k=1}^s q_ku^{-1}X_{w_k}\cap Q/P\right|=\left|\bigcap_{k=1}^s q_kX_{v_k}\right|=c_v\end{equation} where we choose $q_k\in Q$ such that $g^{-1}g_kX_{w_k}=q_ku_k^{-1}X_{w_k}$.

\smallskip

If $c_w=0$, then $\bigcap_{k=1}^s g_kX_{w_k}=\emptyset.$   Equation \eqref{equation_cap} implies that $c_v=0$ and hence $c_w=c_u\cdot c_v.$

\smallskip

If $c_w\neq 0$, then we have a surjection $$\pi\left(\bigcap_{k=1}^s g_kX_{w_k}\right)=\bigcap_{k=1}^s g_kX_{u_k}.$$  Equation \eqref{equation_cap} again implies that $c_w=c_u\cdot c_v$.

\smallskip

Finally, if $c_u=0$, then $c_w=0$ since for generic $(\bar g_1\ldots,\bar g_s)\in (G/B)^s$, we have $$\pi\left(\bigcap_{k=1}^s g_kX_{w_k}\right)\subseteq\bigcap_{k=1}^s g_kX_{u_k}=\emptyset.$$  Hence we still have $c_w=c_u\cdot c_v$. \hfill$\Box$

\section{Applications to Levi-movability}

One application of Theorem \ref{Thm1} is to compute structure coefficients corresponding to Levi-movable $s$-tuples in $(W^P)^s$.  We begin with some preliminaries on Lie theory.  Let $\Delta = \{\alpha_1,\alpha_2,\ldots,\alpha_{n}\}\subset R^+$ be the set of simple roots of $G$ where $n$ is the rank of $G$.  Note that the set $\Delta$ forms a basis for $\gothH^*$ and let $\{x_1,x_2,\ldots,x_n\}\subseteq \gothH$ be the dual basis to $\Delta$ such that $$\alpha_i(x_j)=\delta_{i,j}.$$  Let $\Delta(P)\subset\Delta$ denote the simple roots associated to $P$  (the simple roots that generate $R_P^+$).  For any parabolic subgroup $P$ and $w\in W^P$, define the character $$\chi^P_w:=\rho-2\rho^P+w^{-1}\rho$$ where $\rho$ is the half sum of all the roots in $R^+$ and $\rho^P$ is the half sum of roots in $R^+_P$.  The following proposition is proved in \cite{BK06} using geometric invariant theory:

\begin{proposition}\label{Bel_Ku Thm} (Belkale and Kumar \cite[Theorem 15]{BK06}) If $(w_1,\ldots, w_s)$ is $L_P$-movable, then for every $\alpha_i\in\Delta\backslash\Delta(P)$, we have $$\big((\sum_{k=1}^s \chi^P_{w^k})-\chi^P_1\big)(x_i)=0.$$\end{proposition}

\subsection{Proof of Theorem \ref{Thm2}}

Recall that by Lemma \ref{lemma_wuv}, for any $w\in W^P$, we have $w=uv$ such that $u\in W^Q$ and $v\in W^P\cap W_Q$.  For any pair of parabolic subgroups $P\subseteq Q$, let $T^P:=\Te(G/P)$ and $T^{P,Q}:=\Te(Q/P)$.  For any $w\in W^P$ and $p\in P$ we have the subspace $pT^P_w:=\Te(pw^{-1}X_w)\subseteq T^P$.  The condition for Levi-movability is equivalent to the condition that the diagonal map $$\phi:T^P\rightarrow\bigoplus_{k=1}^sT^P/l_kT^P_{w_k}$$ is an isomorphism for generic $(l_1,\ldots,l_s)\in (L_P)^s$.  Consider the diagram
\begin{equation}\label{comm_tangent}\xymatrix{\quad\, T^{P,Q}\ \ar@{^{(}->}[r] \ar[d]^-{\phi_2}&\quad T^P\ \ar@{->>}[r] \ar[d]^-{\phi} &\quad T^Q\ \ar[d]^-{\phi_1}\\
\displaystyle\bigoplus_{k=1}^s \frac{T^{P,Q}}{l_kT^{P,Q}_{v_k}}\ \ar@{^{(}->}[r] &\ \displaystyle\bigoplus_{k=1}^s \frac{T^P}{l_kT^P_{w_k}}\ \ar@{->>}[r]&\ \displaystyle\bigoplus_{k=1}^s \frac{T^Q}{l_kv_k^{-1}T^Q_{u_k}}}\end{equation} where $\phi_1$ and $\phi_2$ are the diagonal maps corresponding to $G/Q$ and $Q/P$.  It suffices to show that if $\phi$ is an isomorphism, then $\phi_1$ and $\phi_2$ are isomorphisms.

\smallskip

Fix $(l_1,\ldots,l_s)\in (L_P)^s$ so that $\phi$ is an isomorphism.  By the commutativity of the diagram \eqref{comm_tangent}, $\dim\coker\phi_1=0,$ since $\dim\coker\phi=0$.  If we assume that $\dim\ker\phi_1=0$, then $\phi_1$ is an isomorphism which proves part (1).  Since $\phi$ is injective, $\phi_2$ is also injective. By the snake lemma, we have that $$\dim\ker\phi_1=\dim\coker\phi_2=0.$$  Hence $\phi_2$ is an isomorphism which proves part (2).  We now prove that $\dim\ker\phi_1=0$.  Since $\phi_1$ is surjective, the map $$\phi_1: T^Q/\ker{\phi_1}\rightarrow \bigoplus_{k=1}^s\frac{T^Q}{l_iv_k^{-1}T^Q_{u_k}}$$ is an isomorphism.  As a consequence, the induced map on top exterior powers:
$$\Phi_1:\det(T^Q/\ker{\phi_1})\rightarrow \det(\bigoplus_{k=1}^s\frac{T^Q}{l_kv_k^{-1}T^Q_{u_k}})$$ is nonzero.  Identifying the character group $X(H)$ with the weight lattice in $\gothH^*$ we have that $\gothH$ acts on the complex line $\det(T^Q/\ker{\phi_1})$ by the character $-\chi^Q_1-\beta$ where $\beta$ is the sum of roots in $\ker\phi_1$.  Similarly, we have that $\gothH$ acts diagonally on $\det(\displaystyle\bigoplus_{k=1}^s\frac{T^Q}{l_kv_k^{-1}T^Q_{u_k}})$ by the character $\displaystyle -\sum_{i=1}^s\chi^Q_{u_i}.$  It is easy to see that the map $\Phi_1$ is equivariant with respect to the action of the center of $L_Q$.  In particular, for any $\alpha_i\in\Delta\backslash\Delta(Q)$, we have $$(\chi^Q_1+\beta)(x_i)=\sum_{k=1}^s\chi^Q_{u_k}(x_i).$$  For any $w=uv\in W^P$ and $\alpha_i\in\Delta\backslash\Delta(Q)$, we have
\begin{eqnarray*}\chi_w^P(x_i)&=&(\rho-2\rho^P)(x_i)+w^{-1}\rho(x_i)\\ &=&\rho(x_i)-\rho(uv x_i)\\ &=&(\rho-2\rho^Q)(x_i)+u^{-1}\rho(x_i)\\ &=&\chi_u^Q(x_i)\end{eqnarray*} since the Weyl group $W_Q$ acts trivially on $x_i$ and $\rho^P(x_i)=\rho^Q(x_i)=0$.  Hence, by Proposition \ref{Bel_Ku Thm}, we have $$\beta(x_i)=\big((\sum_{k=1}^s\chi^Q_{u_i})-\chi^Q_1\big)(x_i)=\big((\sum_{i=1}^s\chi^P_{w_i})-\chi^P_1\big)(x_i)=0$$  for all $\alpha_i\in\Delta\backslash\Delta(Q).$  But $$\ker\phi_1\subseteq T^Q=\bigoplus_{\alpha\in R^-\backslash R^-_Q} \gothG_{\alpha}$$ where $\gothG_{\alpha}$ denotes the root space of $\gothG$ corresponding to $\alpha.$   Hence $-\beta$ is a positive linear combination of positive simple roots in $\Delta\backslash\Delta(Q).$  This implies that $\dim\ker\phi_1=0.$  This proves Theorem \ref{Thm2}.\hfill$\Box$

\section{Branching Schubert calculus}\label{Branching}

In this section we generalize Theorems \ref{Thm1} and \ref{Thm2} to the setting of branching Schubert calculus.  These generalizations are stated in Theorems \ref{Thm1branching} and \ref{Thm2branching}.  Since the proofs are similar to those for Theorems \ref{Thm1} and \ref{Thm2}, we leave several details to the reader.  Let $\tilG$ be any connected semisimple subgroup of $G$ and fix maximal tori and Borel subgroups $\tilH\subseteq\tilB\subseteq\tilG$ and $H\subseteq B\subseteq G$ such that $\tilH=H\cap\tilG$ and $\tilB=B\cap\tilG$.  As in Theorem \ref{Thm1}, we consider a pair of parabolic subgroups $P\subseteq Q$ in $G$ which contain $B$.   Define parabolic subgroups $$\tilP:=P\cap\tilG\quad \mbox{and}\quad \tilQ:=Q\cap\tilG$$ and consider the maps $$\phi:\tilG/\tilP\hookrightarrow G/P$$ $$\phi_1:\tilG/\tilQ\hookrightarrow G/Q$$ $$\phi_2:\tilQ/\tilP\rightarrow Q/P$$ defined by $\phi(g\tilP):=gP$, $\phi_1(g\tilQ):=gQ$ and $\phi_2:=\phi|_{\tilQ/\tilP}$.  Consider the following commutative diagram:

\begin{equation}\label{branching commute}\xymatrix{\tilQ/\tilP\ \ar@{^{(}->}[r] \ar[d]^-{\phi_2}&\ \tilG/\tilP\ \ar@{->>}[r]^-{\pi} \ar[d]^-{\phi} &\ \tilG/\tilQ\ar[d]^-{\phi_1}\\
Q/P\ \ar@{^{(}->}[r] &\ G/P\ \ar@{->>}[r]&\ G/Q}\end{equation}

For any $w\in W^P$ such that $\dim X_w=\dim G/P-\dim\tilG/\tilP$, we have the associated structure constant $c_w\in\mathds{Z}_{\geq 0}$ defined by $$\phi^*([X_w])=c_w[pt].$$  By Lemma \ref{lemma_wuv}, we can write $w=uv$ where $u\in W^Q$ and $v\in W^P\cap W_Q$.  We show that if condition \eqref{codim condition2} is satisfied, then $c_w=c_u\cdot c_v$ where $$\phi_1^*([X_u])=c_u[pt]\quad \mbox{and}\quad \phi_2^*([X_v])=c_v[pt].$$

\subsection{Proof of Theorem \ref{Thm1branching}}

If $w\in W^P$ satisfies condition \eqref{codim condition2}, then there exists a nonempty open subset $O_1\subseteq G/B$, such that if $\bar g \in O_1$, then the cardinality of inverse images $$|\phi^{-1}(gX_w)|=c_w\quad \mbox{and}\quad |\phi_1^{-1}(gX_u)|=c_u.$$  Consider the projection $\pi:\tilG/\tilP\rightarrow\tilG/\tilQ$.  By the commutativity of diagram \eqref{branching commute}, we have that $\pi(\phi^{-1}(gX_w))\subseteq\phi_1^{-1}(gX_u)$.  Hence if $c_u=0$, then $c_w=0$.

\smallskip

Assume that $c_u\neq 0$.  It suffices to show that for generic $\bar g\in G/B$, the map $\pi$ restricted to $\phi^{-1}(gX_w)$ is surjective when $c_w\neq 0$ and for any $\bar h\in\phi_1^{-1}(gX_u)$, we have $|\pi^{-1}(\bar h)\cap\phi^{-1}(gX_w)|=c_v$.  Following the proof of Theorem \ref{Thm1}, we define the set $$Y:=\{(\bar h,\bar g)\in \tilG/\tilQ\times G/B\ |\ \phi(\bar h)\in gX_u\}.$$  By an analogue of Lemma \ref{Lemma Ysmooth}, the set $Y$ is a smooth irreducible $\tilG$-variety.   Similarly, by an analogue of Proposition \ref{Prop P1 open}, the set of points in $Y$ with the following property $P2$ contains a nonempty open subset of $Y$.

\begin{definition}We say $(\bar h,\bar g)\in Y$ has property $P2$ if:
\begin{enumerate}
\item The intersection $(h^{-1}gX_w\cap Q/P)\cap \phi_2(\tilQ/\tilP)$ is transverse at every point in $Q/P.$
\item For any $q\in Q$ such that $h^{-1}gX_u=qu^{-1}X_u\subseteq G/Q$, the intersection $$qX_v\cap\phi_2(\tilQ/\tilP)=q\bar X_v\cap \phi_2(\tilQ/\tilP)\subseteq Q/P.$$\end{enumerate}\end{definition}

Let $Y^{\circ}\subseteq Y$ be a nonempty open set whose points have property $P2$ and let $\sigma:Y\rightarrow G/B$ denote the projection onto the second factor of $Y$.  By the proof of Theorem \ref{Thm1}, the set $O_2:=(G/B)\backslash\overline{\sigma(Y\backslash Y^{\circ})}$ is an open subset of $G/B$.  Moveover, if $g\in O_1\cap O_2$ and $c_w\neq 0$, then $\pi(\phi^{-1}(gX_w))=\phi_1^{-1}(gX_u)$. By \cite[Lemma 1]{BK06}, we can choose $q\in Q$ such that $h^{-1}gX_w=qu^{-1}X_w$.  By Lemma \ref{Lemma_intersection}, for any $\bar h\in\phi_1^{-1}(gX_u)$, we have \begin{equation}\label{equation_cap2}|\pi^{-1}(\bar h)\cap\phi^{-1}(gX_w)|=|qu^{-1}X_w\cap Q/P\cap \phi_2(\tilQ/\tilP)|=|qX_v\cap\phi_2(\tilQ/\tilP)|=c_v.\end{equation}  If $c_w=0$, then equation \eqref{equation_cap2} implies that $c_v=0$.  In either case, $c_w=c_u\cdot c_v$.  This proves Theorem \ref{Thm1branching}.\hfill$\Box$

\subsection{Proof of Theorem \ref{Thm2branching}}

Let $\tilR$ denote the set of roots of $\tilG$ with respect to the torus $\tilH$ and let $\tilR^+$ denote the set of positive roots with respect to the Borel $\tilB$.  Let $\Delta_{\tilG}:=\{\tilde\alpha_1,\ldots,\tilde\alpha_m\}\subseteq\tilR^+$ denote the of simple roots of $\tilG$ where $m$ is the rank of $\tilG$.  Let $\{\tilde x_1,\ldots,\tilde x_m\}\subseteq\tilde\gothH$ denote the dual basis to $\Delta_{\tilG}$.  For any parabolic subgroup $\tilQ\subseteq \tilG$ which contains $\tilB$, let $\tilR^+_{\tilQ}$ denote the positive roots of of $\tilQ$ or $L_{\tilQ}$ and let $\Delta_{\tilQ}:=\Delta_{\tilG}(\tilQ)\subseteq\Delta_{\tilG}$ denote the corresponding simple roots.  Consider the following diagram which is analogous to \eqref{comm_tangent}.  By an abuse of notation we will use $\phi, \phi_1,\phi_2$ to denote the induced map on Lie algebras.
\begin{equation}\label{comm_tangent_branching}\xymatrix{\quad\, \tilde T^{P,Q}\ \ar@{^{(}->}[r] \ar[d]^-{\phi_2}&\quad \tilde T^P\ \ar@{->>}[r] \ar[d]^-{\phi} &\quad \tilde T^Q\ \ar[d]^-{\phi_1}\\ \displaystyle\frac{T^{P,Q}}{lT^{P,Q}_v}\ \ar@{^{(}->}[r] &\  \displaystyle \frac{T^P}{lT^P_w}\ \ar@{->>}[r]&\ \displaystyle\frac{T^Q}{lv^{-1}T^Q_u}}\end{equation}
Since $w\in W^P$ is $(L_P,\phi)$-movable, the map $\phi$ is an isomorphism for general $l\in L_P$.  By the snake lemma, it suffices to show that $\phi_1$ is injective.  Let $\beta\in\tilde\gothH^*$ denote the sum of roots corresponding to $\ker\phi_1$.  Following the proof of Theorem \ref{Thm2}, it suffices to show that $\beta(\tilde x_i)=0$ for all $\tilde\alpha_i\in\Delta_{\tilG}\backslash\Delta_{\tilQ}$ since $\ker\phi_1\subseteq \tilde T^Q$.  Consider the group $$C:=\tilH\cap Z(L_Q)$$ where $Z(L_Q)$ denotes the center of $L_Q$. Observe that $C\subseteq Z(L_{\tilQ})$ and that $\Lie(C)=\tilde\gothH\cap \mathfrak{Z}$ where $\mathfrak{Z}$ denotes the Lie algebra of $Z(L_Q)$.   Since $C\subseteq\tilH$, we have induced $C$-module structures on $\tilde T^P, \tilde T^Q, \tilde T^{P,Q}$.  Similarly, since $C\subseteq Z(L_Q)$, we have induced $C$-module structures on $T^P, T^Q, T^{P,Q}.$  It is easy to see that the maps $\phi, \phi_1$ and $\phi_2$ are $C$-equivariant.  Since $\phi$ is an isomorphism and $\phi_1$ is surjective, the induced $C$-equivariant maps $$\Phi:\det(\tilde T^P)\rightarrow \det(T^P/lT^P_w)$$ and $$\Phi_1:\det(\tilde T^Q/\ker\phi_1)\rightarrow \det(T^Q/lv^{-1}T^Q_u)$$ are nonzero.  Let $i:\tilG\hookrightarrow G$ denote the embedding of $\tilG$ into $G$ and define the character $$\tilde\chi^{\tilP}:= 2(\tilde\rho-\tilde\rho^{\tilP})$$ where $\tilde\rho$ is the half sum of all roots in $\tilR^+$ and $\tilde\rho^{\tilP}$ is the half sum of all roots in $\tilR^+_{\tilP}$.  We have that $\tilde\gothH$ acts on $\det(\tilde T^P)$ by the character $-\tilde\chi^{\tilP}$.  For any $\tau\in\Lie(C)$ we have $$\beta(\tau)=\big(i^*\chi^Q_u-\tilde\chi^{\tilQ}\big)(\tau)=\big(i^*\chi^P_w-\tilde\chi^{\tilP}\big)(\tau)=0$$
since the isomorphisms $\Phi$ and $\Phi_1$ are $C$-equivariant.  By assumptions in Theorem \ref{Thm2branching}, there exists a vector $\tau_0\in\Lie(C)$ such that $\alpha(\tau_0)\geq 0$ for any $\alpha\in\Delta_{\tilG}$ with equality if and only if $\alpha\in\Delta_{\tilQ}$.  This implies that $\beta(\tilde x_i)=0\quad \forall\ \tilde\alpha_i\in\Delta_{\tilG}\backslash\Delta_{\tilQ}$ and hence $\dim\ker\phi_1=0$.  This proves Theorem \ref{Thm2branching}.\hfill$\Box$

\bibliographystyle{abbrv}
\bibliography{references}

\end{document}